\documentclass{article}

\usepackage{amsmath,amssymb,amsthm,hyperref}
\usepackage{indentfirst,float}
\usepackage{amsmath,amsfonts,amssymb,stackrel,amsthm,bbm,amssymb,hyperref,array,enumitem}
\usepackage{geometry,centernot,longtable,makeidx,xcolor,mathtools,xpatch,multirow,graphicx}
\newtheorem{lemma}{Lemma}[section]
\newtheorem{definition}[lemma]{Definition}

\newtheorem{corollary}[lemma]{Corollary}
\newtheorem{theorem}{Theorem}
\newtheorem{observation}{Observation}

\newcommand\ex{{\mathrm{ex}}}
\newcommand\eex{{\mathrm{Ex}}}

\title{On edge-ordered graphs with linear extremal functions\thanks{A preliminary version of this paper appeared in the proceedings of EUROCOMB'23, \cite{EC23}}}
\author{Gaurav Kucheriya \thanks{Department of Applied Mathematics, Charles University, Prague, Czechia, Email: \href{mailto:gaurav@kam.mff.cuni.cz}{gaurav@kam.mff.cuni.cz}. Supported by GA\v{C}R grant 22-19073S and SVV–2023–260699.} \and G\'abor Tardos \thanks{Alfr\'ed R\'enyi Institute of Mathematics, Budapest, Hungary, Email: \href{mailto:tardos@renyi.hu}{tardos@renyi.hu}. Supported by
the ERC advanced grants ERMiD and GeoScape and the National Research, Development and Innovation
Office (NKFIH) grants K-132696 and SNN-135643}}
\date{}
\begin{document}
\maketitle
\begin{abstract}
The systematic study of Tur\'an-type extremal problems for edge-ordered graphs was initiated by Gerbner et al.\ in 2020. Here we characterize connected edge-ordered graphs with linear extremal functions and show that the extremal function of other connected edge-ordered graphs is $\Omega(n\log n)$. This characterization and dichotomy are similar in spirit to results of F\"uredi et al.\ (2020) about vertex-ordered and convex geometric graphs. We also extend the study of extremal function of short edge-ordered paths by Gerbner et al.\ to some longer paths.
\end{abstract}

\section{Introduction}

Tur\'an-type extremal graph theory asks how many edges an $n$-vertex simple graph can have if it does not contain a subgraph isomorphic to a \emph{forbidden graph}. We introduce the relevant notation here.

\begin{definition}
We say that a simple graph $G$ \emph{avoids} another simple graph $H$, if no subgraph of $G$ is isomorphic to $H$. The Tur\'an number $\ex(n,H)$ of a \emph{forbidden} finite simple graph $H$ (having at least one edge) is the maximum number of edges in an $n$-vertex simple graph avoiding $H$.
\end{definition}

This theory has proved to be useful and applicable in combinatorics, as well as in combinatorial geometry, number theory and other parts of mathematics and theoretical computer science.

Tur\'an-type extremal graph theory was later extended in several directions, including hypergraphs, geometric graphs, convex geometric graphs, vertex-ordered graphs, etc. Here we work with \emph{edge-ordered graphs} as introduced by Gerbner, Methuku, Nagy, P\'alv\"olgyi, Tardos and Vizer in \cite{GMNPTV}. The several extensions of extremal graph theory each proved useful and applicable in different parts of mathematics and this also holds for the (still new) edge-ordered version discussed here, see e.g. \cite{KP}. Let us recall the basic definitions.

\begin{definition}
An \emph{edge-ordered graph} is a finite simple graph $G$ together with a linear order on its edge set $E$. We often give the edge-order with an injective labeling $L: E\to\mathbb R$. We denote the edge-ordered graph obtained this way by $G^L$, in which an edge $e$ precedes another edge $f$ in the edge-order (denoted by $e<f$) if $L(e)<L(f)$. We call $G^L$ the \emph{labeling} or \emph{edge-ordering} of $G$ and call $G$ the simple graph \emph{underlying} $G^L$.

An isomorphism between edge-ordered graphs must respect the edge-order. A subgraph of an edge-ordered graph inherits the edge-order and so it is also an edge-ordered graph. We say that the edge-ordered graph $G$ \emph{contains} another edge-ordered graph $H$, if $H$ is isomorphic to a subgraph of $G$ otherwise we say that $G$ \emph{avoids} $H$.

For a positive integer $n$ and an edge-ordered graph $H$, let the Tur\'an number $\ex_<(n,H)$ be the maximal number of edges in an edge-ordered graph on $n$ vertices that avoids $H$. Fixing the \emph{forbidden edge-ordered graph} $H$, $\ex_<(n, H)$ is a function of $n$ and we call it the \emph{extremal function} of $H$. Note that this definition does not make sense if $H$ has no edges, so we insist that $H$ is \emph{non-trivial}, that is, it has at least one edge.
\end{definition}

Bra\ss, K\'arolyi and Valtr, \cite{BKV} introduced \emph{convex geometric graphs} while Pach and Tardos, \cite{PT} introduced \emph{vertex-ordered graphs} and studied their extremal theories. In both cases a simple graph is given extra structure by specifying an order on their vertices (a cyclic order for convex geometric graphs and a linear order for vertex-ordered graphs). Characterizing the convex geometric or vertex-ordered graphs with a linear extremal function seems to be beyond reach (so far), but F\"uredi, Kostochka, Mubayi and Verstra\"ete, \cite{FKMV} found such a characterization for \emph{connected} convex geometric graphs and also for \emph{connected} vertex-ordered graphs. The situation seems to be similar for edge-ordered graphs: while we could not give a general characterization of edge-ordered graphs with linear extremal functions, in Section~\ref{sec2}  we characterize when connected edge-ordered graphs have linear extremal functions. This characterization is also a dichotomy result: we show that whenever the extremal function of a connected edge-ordered graph is not linear, it must be $\Omega(n\log n)$.

Gerbner et al., \cite{GMNPTV} estimated the extremal functions of many small edge-ordered graphs, among them all edge-ordered paths with up to four edges. In Section~\ref{sec4} we extend these results to some 5-edge paths. Lastly, we provide some remarks and open problems in Section~\ref{sec5}.

\section{Connected edge-ordered graphs with linear extremal functions}\label{sec2}

In classical (unordered) extremal graph theory the following dichotomy is immediate:

\begin{observation}\label{simple}
If $H$ is a forest, then $\ex(n,H)=O(n)$, otherwise $\ex(n,H)=\Omega(n^c)$ for some $c=c(H)>1$.
\end{observation}

The analogous statement fails for edge-ordered graphs. For example, the paper \cite{GMNPTV} exhibits several edge-ordered paths with extremal functions $\Theta(n\log n)$ (see Theorem~\ref{34edgepath} below). Therefore, when looking for an analogous result for edge-ordered graphs, we have a choice to make. Either we want to characterize the edge-ordered graphs with linear extremal functions, or the ones with extremal functions that is \emph{almost linear}, i.e., $n^{1+o(1)}$. In the latter direction Gerbner et al.\ \cite{GMNPTV} formulated a conjecture that we recently verified, see \cite{KT}. The former problem seems to be considerably more difficult as there is not even a reasonable conjecture characterizing all edge-ordered graphs with a linear extremal function. 

The first result in this harder direction appeared in the MSc thesis of the first author, \cite{Kthesis}: he gave a simple characterization of edge-ordered \emph{paths} with linear extremal functions. In this section we generalize this result and provide a characterization for \emph{connected} edge-ordered graphs with linear extremal functons, see Theorem~\ref{dich}. This is the same restriction considered by F\"uredi et al., \cite{FKMV} with respect to vertex-ordered and convex geometric graphs.

Our theorem also states that if the extremal function of a connected edge-ordered graph is not linear, then it is $\Omega(n\log n)$. Such a dichotomy does not hold for edge-ordered graphs in general as Gerbner et al.\ also exhibit a (necessarily disconnected) edge-ordered graph whose extremal function is $\Theta(n\alpha(n))$, where $\alpha$ is the inverse of the Ackermann function.
\medskip

In order to formulate the characterization and dichotomy in Theorem~\ref{dich} we need to introduce some terminology. The \emph{order chromatic number} $\chi_<(G)$ of an edge-ordered graph $G$ is the smallest chromatic number $\chi(H)$ of a simple graph $H$ such that all edge-orderings of $H$ contain $G$. If no such $H$ exists we write $\chi_<(G)=\infty$. The order chromatic number was introduced in the paper \cite{GMNPTV} to play the role of the (ordinary) chromatic number in a version of the Erd\H os-Stone-Simonovits theorem for edge-ordered graphs, see Theorem~2.3 in \cite{GMNPTV}. For the purposes of our Theorem~\ref{dich}, one does not even have to apply this definition, it is enough to apply Lemma~\ref{semi-right} below that gives a simple characterization when the order chromatic number of an edge-ordered forest is two based on the notion of close vertices. We call a vertex $v$ of an edge-ordered graph \emph{close} if the edges adjacent to $v$ form an interval in the edge-order. For example, all blue vertices in Figure~\ref{fig1}, as well as all but two red vertices are close.

\begin{lemma}[\cite{GMNPTV}]\label{ocn2}
	A non-trivial edge-ordered forest has order chromatic number 2 if and only if it has a proper 2-coloring such that all vertices in one of the color classes are close.
\end{lemma}

We call the edges $e_1<e_2$ \emph{consecutive} in an edge-ordered graph $G$ if no edge $e$ of $G$ satisfies $e_1<e<e_2$. An edge-ordered graph $G$ is a \emph{semi-caterpillar} if the underlying simple graph is a non-trivial tree and any pair of consecutive edges in $G$ are either adjacent in $G$ or they are directly connected by an edge larger than both of them. The \emph{reverse} $G^R$ of an edge-ordered graph $G$ is obtained from $G$ by keeping its underlying simple graph and reversing the edge-order. Note that the extremal functions of $G$ and $G^R$ coincide and so does their order chromatic numbers.

\begin{theorem}[Dichotomy]\label{dich}
	If $G$ or its reverse $G^R$ is a semi-caterpillar of order chromatic number 2, then
	$\ex_<(n,G)=O(n)$. For any other non-trivial connected edge-ordered graph $G$ we have $\ex_<(n,G)=\Omega(n\log n)$.
\end{theorem}

Neither direction of the above dichotomy seems to follow from earlier results. For the linear bound for order chromatic number 2 semi-caterpillars we will develop and use a recursive process to obtain all these edge-ordered graphs (or rather their \emph{bipartitions}, see Lemma~\ref{equivalence}). For the superlinear lower bound we will show that all connected order-chromatic number 2 edge-ordered graphs that are \emph{not} semi-caterpillars contain certain edge-ordered paths (see Theorem~\ref{semi}) and then use existing results on the extremal functions of these edge-ordered paths.

An important tool we borrow from the paper~\cite{GMNPTV} is the bipartitions of edge-ordered graphs. An \emph{edge-ordered bigraph} $G$ is an edge-ordered graph $G_0$ together with a proper 2-coloring to \emph{left} and \emph{right} vertices, so each edge has a left end and a right end. We call $G_0$ the edge-ordered graph \emph{underlying} $G$ and we say that $G$ is a \emph{bipartition} of $G_0$. Note that we use many terms, like edge-ordered forest, edge-ordered tree, edge-ordered path in a simpler sense meaning an edge-ordered graph whose underlying simple graph is a forest, a tree, or a path, respectively. Edge-ordered bigraph (as defined above) is more than an edge-ordered bipartite graph (in the above sense) as left and right vertices are distinguished. The notions of \emph{isomorphism}, \emph{subgraph}, \emph{contain}, \emph{avoid} and \emph{close vertex} naturally extend to edge-ordered bigraphs. 

The paper \cite{GMNPTV} introduced bipartitions in order to break the symmetry. Using them one can distinguish the two ways a connected edge-ordered graph may be embedded in another edge-ordered graph if both underlying simple graphs happen to be bipartite: after making them into edge-ordered bigraphs by designating left and right vertices in both graphs either all left vertices map to left vertices and the mapping ensures containment between the edge-ordered bigraphs or all left vertices map to right vertices in which case it does not.\footnote{The paper \cite{GMNPTV} used the terms \emph{edge-ordered bipartite graph} instead of edge-ordered bigraph and the terms \emph{left-contain} and \emph{right-contain} for the two ways an edge-ordered bigraph can contain an edge-ordered path.}

We call an edge-ordered bigraph a \emph{right caterpillar} if its underlying edge-ordered graph is a semi-caterpillar and all its right vertices are close. See Figure~\ref{fig1} for an example of a right caterpillar. The following corollary is a direct consequence of Lemma~\ref{ocn2}.

\begin{corollary}\label{semi-right}
A semi-caterpillar has order chromatic number 2 if and only if one of its bipartitions is a right caterpillar.
\end{corollary}

\begin{figure}[h]
	\centering
	\fbox{\includegraphics[scale=1.2]{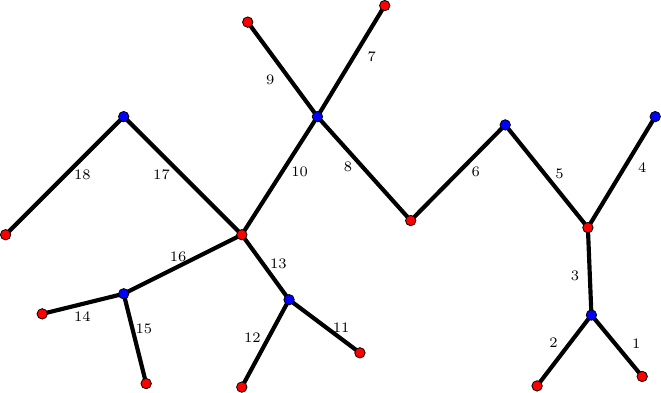}}
	\caption{A right caterpillar. Left vertices are red, right vertices are blue.}
	\label{fig1}
\end{figure}

In the next lemma, we will give an alternative definition for right caterpillars that is often easier to use. But we need a definition first:
Let $G$ be a non-trivial edge-ordered bigraph and let $e$ be the smallest edge in $G$. We call the edge-ordered bigraph $G'$ an \emph{extension} of $G$ if $G'$ is obtained from $G$ by adding new edges to it, such that
\begin{enumerate}
\item every new edge connects one end of $e$ to a new degree 1 vertex;
\item all new edges are smaller than the edge $e$;
\item all new edges incident to the left end of $e$ are smaller than any new edge incident to the right end of $e$.
\end{enumerate}
Let $T_0$ denote the unique (up to isomorphism) edge-ordered bigraph with a single edge and two vertices. 

\begin{lemma}\label{equivalence}
An edge-ordered bigraph is a right caterpillar if and only if it can be obtained (up to isomorphism) from $T_0$ by a sequence of extensions.
\end{lemma}

\begin{proof}
We prove the ``if'' part of the lemma first. $T_0$ is clearly a right caterpillar. Now assume $T'$ is the extension of the right caterpillar $T$ and let $e$ denote the smallest edge of $T$. It is clear that the underlying simple graph of $T'$ is a tree. To see that it is a semi-caterpillar consider two consecutive edges of $T'$. If both are in $T$, then they are adjacent or connected by a larger edge because the edge-ordered graph underlying $T$ is a semi-caterpillar. If one of the two consecutive edges is in $T$, but the other is not, then the former must be the smallest edge in $T$ and therefore the latter must be adjacent to it. Finally, if both consecutive edges are outside $T$, then they are both adjacent to the smallest edge $e$ in $T$, therefore either adjacent to each other or  connected by a larger edge, namely $e$. To see that $T'$ is a right caterpillar we need to further ensure that all right vertices are close in $T'$. This property is inherited from $T$ for all its right vertices except perhaps for the right end $v$ of the smallest edge $e$ in $T$ because their neighborhoods did not change. For $v$ this neighborhood included the smallest edge $e$ in $T$ and now in $T'$ we added a few edges just below $e$, so $v$ also remains close in $T'$. All vertices in $T'$ outside $T$ have degree $1$, so they are also close.
 
Before moving to the ``only if'' part of the lemma, let us note that if one removes a few of the smallest edges of a right caterpillar $T$ (but not all edges) and also the vertices that became isolated, then the remaining edge-ordered bigraph $T'$ is also a right caterpillar. Indeed, any pair of consecutive edges in $T'$ are also consecutive in $T$ and so they are adjacent or connected by a larger edge in $T$ and therefore also in $T'$. Also, $T'$ must remain connected (and thus a tree) as otherwise there were a pair of consecutive edges in $T'$ in different components, failing this condition. Thus the underlying edge-ordered graph of $T'$ is a semi-caterpillar. The property that the right vertices are close in $T$ must hold for every subgraph of $T$, so in particularly for $T'$ too, making it a right caterpillar.

We now prove the ``only if'' part of the lemma by induction on the size of the right caterpillar. If it has a single edge, then it is isomorphic to $T_0$ itself. Therefore, let $T$ be a right caterpillar with at least two edges and let $e_1$ and $e_2$ be the smallest and second smallest among them, respectively.

First we consider the case that $e_1$ and $e_2$ are adjacent. Remove $e_1$ from $T$ together with the isolated vertex created and call the remaining right caterpillar $T'$. $T'$ can be obtained from $T_0$ by a sequence of extension by the inductive hypothesis. But $T$ is an extension of $T'$ (by a single edge), so $T$ can also be obtained from $T_0$ by a sequence of extensions.

Lastly, we consider the case when the consecutive edges $e_1$ and $e_2$ are not adjacent and let $e$ be the edge connecting them. Let us obtain $T'$ from $T$ by removing all edges smaller than $e$ (including both $e_1$ and $e_2$) and all resulting isolated vertices. We saw that $T'$ is a right caterpillar and thus it can be obtained from $T_0$ by a sequence of extensions. It remains to show that $T$ is an extension of $T'$.

The vertex where $e$ and $e_1$ meet is not close since $e_1<e_2<e$ and $e_2$ is not incident to it, so it must be a left vertex. This makes the vertex where $e$ and $e_2$ meet a right vertex and therefore close in $T$ and thus incident to all the edges between $e$ and $e_2$, that is, all edges outside $T'$ except for $e_1$. Note also that the ends of all removed edges not on $e$ must be degree $1$ vertices in $T$ because otherwise these edges would create a cycle with the tree $T'$. This makes $T$ an extension of $T'$ where we added a single new edge $e_1$ to the left end of the smallest edge $e$ in $T'$ and a few edges to its right end. This finishes the proof of the lemma.
\end{proof}

We call the minimum number of extension steps needed to obtain a right caterpillar $T$ from $T_0$ the \emph{recursive depth} of $T$. The recursive depth of the right caterpillar in Figure~\ref{fig1} is 9.

Note that extensions $G'$ of $G$ introducing a single new edge (or even several new edges all incident to the same end of the smallest edge in $G$) are easy to deal with. Lemma~\ref{add} below shows that in this case the extremal function of the underlying edge-ordered graph increases by additive linear term only. Unfortunately, this fails to hold for general extensions as defined above. To bound the extremal function, we need to deal with the entire sequence of extensions at once and for that we need the following definitions. 

Let $ G $ be an edge-ordered bigraph on $ n $ vertices and $c$ be a positive integer. We define an edge $ e $ of $ G $ to be \emph{$c$-left-leaning} if there exist sets $ S $ and $ S' $ of $ c $ edges each in $ G $ such that
\begin{itemize}	\itemsep-5pt 
	\item all edges in $ S $ are incident to the left end of $ e $,
	\item all edges in $ S' $ are incident to the right end of $ e $ and
	\item all edges in $ S' $ are larger than any edge in $ S $ but smaller than $ e $.
\end{itemize}

Similarly, we define $e$ to be \emph{$c$-right-leaning} in $G$, if size $c$ edge sets $S$ and $S'$ as above exist but now the edges in $S$ should be larger than the edges in $S'$ but still smaller than $e$. Note that an edge can be both $c$-left-leaning and $c$-right-leaning at the same time, just one of them or neither. We call an edge \emph{$c$-non-leaning} in $G$, if it is neither $c$-left-leaning nor $c$-right-leaning. See Figure~\ref{fig2} for an example.

\begin{figure}[h]
	\centering
	\fbox{\includegraphics[scale=1]{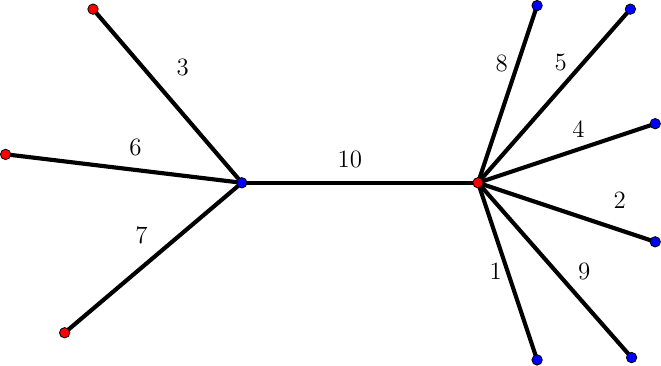}}
	\caption{The edge labeled 10 is a 2-left and 2-right-leaning edge; it is a 3-non-leaning edge.}
	\label{fig2}
\end{figure}

\begin{lemma}\label{2cn}
	For any $ c\ge1 $ any edge-ordered bigraph $ G $ on $ n $ vertices has at most $ 2 c n $ $c$-non-leaning edges.
\end{lemma}
\begin{proof}
	Consider an edge $ e $ that is not among the $2c-1$ smallest edges at either of its endpoints. We select $2c-1$ edges incident to the left end of $e$, all smaller than $e$. Let us call them $ e_1, e_2, \dots, e_{2c-1} $ in increasing order. Similarly, we find the edges $ e'_1, e'_2, \dots, e'_{2c-1} $, also in increasing order, all smaller than $e$, but incident to the right end of $e$.
	
	If $e_c<e'_c$, then we set $ S = {\{e_1, \dots, e_c}\} $, $ S' = {\{e'_c, \dots, e'_{2c-1}}\} $ showing that $ e $ is a $ c $-left-leaning edge. Whereas if $e_c>e'_c $, we have the edge sets $ S = {\{e_c, \dots, e_{2c-1}}\} $, $ S' = {\{e'_1, \dots, e'_{c}}\} $ showing that $ e $ is a $ c $-right-leaning edge.
	
	Thus, all $c$-non-leaning edges in $G$ must be among the $2c-1$ smallest edges incident to some vertex of $G$. This proves the statement of the lemma.
\end{proof}

Let us fix $ c $ and define two sequences of subgraphs for an edge-ordered bigraph $ G $ starting at $ G^{c\textrm{-left}}_0 = G^{c\textrm{-right}}_0 = G $. For $i>0$ we define $ G^{c\textrm{-left}}_i $ as the subgraph of $ G^{c\textrm{-left}}_{i-1} $ consisting of its $c$-left-leaning edges, whereas $ G^{c\textrm{-right}}_i $ is the subgraph of $ G^{c\textrm{-right}}_{i-1} $ consisting of its $ c $-right-leaning edges. All of these are subgraphs of $G$ and contain all vertices of $G$. Notice that if $H$ is a subgraph of $G$, then $c$-left-leaning (respectively, $c$-right-leaning) edges of $H$ are also $c$-left-leaning (respectively, $c$-right-leaning) in $G$, therefore we also have that $G^{c\textrm{-left}}_i$ contains $H^{c\textrm{-left}}_i$ and $G^{c\textrm{-right}}_i$ contains $H^{c\textrm{-right}}_i$ for any $i$. 

\begin{lemma}\label{riglef}
    Let $c,i\ge1$ and let $ G $ be an edge-ordered bigraph on $n$ vertices. There are at most $2i^2cn$ edges in $G$ neither contained in $ G^{c\textrm{\rm-left}}_i $ nor in $ G^{c\textrm{\rm-right}}_i $.
\end{lemma}
\begin{proof}
Consider an edge $e$ of $G$ contained in neither $ G^{c\textrm{-left}}_i $ nor $ G^{c\textrm{-right}}_i $. As $e$ is contained in $G^{c\textrm{-left}}_0=G$ we can find a value $0\le j<i$ such that $e$ is contained in $ G^{c\textrm{-left}}_j $ but not in $ G^{c\textrm{-left}}_{j+1} $. Note that this means that $e$ is not $c$-left-leaning in $G^{c\textrm{-left}}_j$. The graph $(G^{c\textrm{-left}}_j)^{c\textrm{-right}}_i$ is a subgraph of $G^{c\textrm{-right}}_i$, so it does not contain $e$. Therefore, we can find $0\le k<i$ such that $(G^{c\textrm{-left}}_j)^{c\textrm{-right}}_k$ contains $e$ but $(G^{c\textrm{-left}}_j)^{c\textrm{-right}}_{k+1}$ does not. This means that the edge $e$ is not $c$-right-leaning in $(G^{c\textrm{-left}}_j)^{c\textrm{-right}}_k$. It is neither $c$-left-leaning there as it is not $c$-left-leaning in the larger graph $G^{c\textrm{-left}}_j$. So $e$ must be $c$-non-leaning in $(G^{c\textrm{-left}}_j)^{c\textrm{-right}}_k$.

We have just shown that any edge of $G$ not contained in either $ G^{c\textrm{-left}}_i $ or $ G^{c\textrm{-right}}_i $ is a $c$-non-leaning edge of $(G^{c\textrm{-left}}_j)^{c\textrm{-right}}_k$ for some $0\le j,k<i$. By Lemma~\ref{2cn}, we have at most $ 2cn $ $c$-non-leaning edges in any one of these $i^2$ graphs. This proves the lemma.
\end{proof}
\begin{lemma}\label{T}
	Let $ T$ be a right caterpillar with $c$ vertices and recursive depth $ i $ and let $G$ be an edge-ordered bigraph. If $ G^{c\textrm{-left}}_i $ contains an edge, then $ G $ contains $ T $.
\end{lemma}
\begin{proof}
We prove this lemma by induction on $i$, which is the recursive depth of $T$. For $ i = 0 $, that is, when $ T=T_0$, the statement holds trivially. Assume $i\ge1$ and let $T$ be an extension of the right caterpillar $T'$ of recursive depth $i-1$. Let $c'$ be the number of vertices in $T'$. We apply the inductive hypothesis to $T'$ and $G^{c\textrm{-left}}_1$: if $(G^{c\textrm{-left}}_1)^{c'\textrm{-left}}_{i-1}$ contains an edge, then $G^{c\textrm{-left}}_1$ contains $T'$. We clearly have $c'<c$, therefore $(G^{c\textrm{-left}}_1)^{c'\textrm{-left}}_{i-1}$ contains $(G^{c\textrm{-left}}_1)^{c\textrm{-left}}_{i-1}=G^{c\textrm{-left}}_i$. We assume that $ G^{c\textrm{-left}}_i $ contains an edge, so by the inductive hypothesis, a subgraph $T''$ of $G^{c\textrm{-left}}_1$ is isomorphic to $T'$. 

Let $e$ be the smallest edge in $T''$. As $e$ is $c$-left-leaning in $G$, we have $2c$ smaller edges in $G$ with $c$ of them incident to right end of $e$ and $c$ even smaller ones incident to the left end of $e$. We find an isomorphic copy of $T$ in $G$ by adding the appropriate number of these smaller edges to $T''$. The only thing we have to make sure is that all the other ends of the newly added edges should be distinct and outside $T''$. This is doable because we have $c$ smaller edges to choose from at either end of $e$ and we have to avoid less than $c'$ of them that have their other end in $T''$ while we need only $c-c'$ new edges in total. 
\end{proof}

By Lemma~\ref{semi-right} and the fact that the extremal function of an edge-ordered graph and its reverse coincide, the following lemma is equivalent to the linear upper bound stated in Theorem~\ref{dich}.

\begin{lemma}\label{lin}
	We have $ \ex_<(n, T^*) = O(n) $ for the underlying edge-ordered tree $T^*$ of any right caterpillar $T$. 
\end{lemma}
\begin{proof}
	Let $ G $ be an edge-ordered graph on $ n $ vertices having the maximal number of edges $ m = \ex_<(n, T^*) $ that avoids $ T^* $. Consider an edge-ordered bipartite subgraph $ G' $ of $ G $ having at least $ m/2 $ edges and make it an edge-ordered bigraph by appropriately designating left and right vertices. Since $ G $ avoids $ T^* $, $ G' $ avoids $ T $.
	
	Let $c$ be the number of vertices in $T$ and let $i$ be the recursive depth of $T$. Consider $ G'^{c\textrm{-left}}_i $. It must be empty as otherwise $G'$ contains $T$ by Lemma~\ref{T}, a contradiction. Now consider $ G'^{c\textrm{-right}}_i $. By left-right symmetry, if this subgraph of $G' $ is non-empty, then $G'$ contains $ T' $ the edge-ordered bigraph obtained from $T$ by switching the roles of the left and right vertices. But $T'$ is also a bipartition of $T^*$, so this also implies that $G$ contains $T$, another contradiction. Therefore, neither $ G'^{c\textrm{-left}}_i $ nor $ G'^{c\textrm{-right}}_i $ contains any edges. But then $G'$ has at most $2i^2cn$ edges by Lemma~\ref{riglef}. This shows $\textrm{ex}_<(n, T^*)=m\le4i^2cn=O(n)$ as claimed.
\end{proof}

Let us now turn to the other side of the dichotomy stated in Theorem~\ref{dich}. Firstly, we characterize semi-caterpillars in terms of avoiding certain edge-ordered paths. We denote the simple path on $k$ vertices by $P_k$ and denote the edge-ordered path obtained by labeling $ P_k$ by listing the labels along the path in the upper index. So, for example, $ P_4^{213} $ mentioned in the theorem below is an edge-ordered 3-edge path whose middle edge is the smallest.

\begin{theorem}\label{semi}
 An order-chromatic number $2$ edge-ordered graph is a semi-caterpillar if and only if it is connected and does not contain any of the edge-ordered paths $ P_4^{213} $, $ P_5^{1342} $, or $ P_5^{1432} $.
\end{theorem}

\begin{proof} A semi-caterpillar is an edge-ordered tree, so it is connected. If an edge-ordered tree contains $P_4^{213}$, then removing the middle edge $e$ of this path, both resulting connected components contain an edge larger than $e$, so there must be a pair of consecutive edges $e_1<e_2$, both larger than $e$ and in distinct components thus violating the definition of a semi-caterpillar. To finish the proof of the ``only if'' direction of the theorem we need to show that order chromatic number 2 semi-caterpillars contain neither $P_5^{1342}$ nor $P_5^{1432}$. It is not hard to give a direct argument for this but at this point the simplest proof is to refer to the (already proven) half of Theorem~\ref{dich} that the extremal functions of order chromatic number 2 semi-caterpillars are linear and the result from \cite{GMNPTV} that these two edge-ordered paths have super-linear extremal functions (see Theorem~\ref{34edgepath} below), so they cannot be contained in order chromatic number 2 semi-caterpillars.

We now turn to the ``if'' part of theorem and consider a connected edge-ordered graph $G$ of order chromatic number 2 avoiding the three edge-ordered paths listed in the statement of the theorem. Note first that just order chromatic number 2 and avoiding $P_4^{213}$ already imply that $G$ is cycle-free. Indeed, if $G$ contained a cycle, it could not be a triangle because the order chromatic number of an edge-ordered triangle is three, and thus the subgraph formed by the smallest edge of the cycle and its two neighboring edges would be isomorphic to $P_4^{213}$.

Let $e_1<e_2$ be two consecutive edges of $G$. To finish the proof we need to show that they are adjacent or connected by a single edge larger than them. As $G$ is an edge-ordered tree, there is a unique path $P$ connecting $e_1$ and $e_2$ in $G$. Here $e_1$ and $e_2$ are the first and last edges of $P$. No intermediate edge $e$ of this path can be smaller than its two neighbors on $P$ as otherwise these three edges would form a subgraph isomorphic to $P_4^{213}$, a contradiction. This implies that smallest edge of $P$ cannot be an intermediate edge, so $e_1$ is the smallest edge on $P$. Consider the two largest edges $e_3<e_4$ of $P$. They must be adjacent as otherwise the smallest edge between them would violate the observation above. If $e_3$ and $e_4$ are adjacent intermediate edges they, together with their two adjacent edges on $P$ would form a subgraph of $G$ isomorphic to either $P_5^{1342}$ or $P_5^{1432}$, also a contradiction.
	
	There are only two remaining possibilities for the path $P$: either $e_1$ and $e_2$ are adjacent or they are connected by a single edge $e_4>e_2$ (making $e_2=e_3$). This shows that $G$ is a semi-caterpillar as claimed and finishes the proof of the theorem.
\end{proof}

Recall that we simply denote the edge-ordered path $ P_k^L $ by listing the labels along the path in the upper index. We need one last piece of notation for the proof of Theorem~\ref{dich}: to denote a bipartition of an edge-ordered path we simply add a sign $+$ or $-$ in front of the upper index to signify that the path starts at a right or a left vertex, respectively. So $P_4^{+132}$ and $P_4^{-132}$ are the two bipartitions of the edge-ordered path $P_4^{132}$.

\begin{proof}[Proof of Theorem~\ref{dich}.]
As we have remarked above, Lemmas~\ref{semi-right} and \ref{lin} imply the linear bound for the extremal functions of connected order chromatic number 2 semi-caterpillars and their reverses. It remains to prove that if a non-trivial connected edge-ordered graph $G$ does not fall into either of these two categories, then $\ex_<(n,G)=\Omega(n\log n)$. Let $G$ be such an edge-ordered graph.

By the non-triviality, the order-chromatic number of $G$ is not $1$, so it is either 2 or we have $\ex_<(n,G)=\Theta(n^2)=\Omega(n\log n)$ by Theorem~2.3 in \cite{GMNPTV} (the Erd\H os-Stone-Simonovits theorem for edge-ordered graphs). We will therefore assume that the order chromatic number of $G$ is 2.

By Theorem~\ref{semi}, as the non-trivial, connected, order chromatic number 2 edge-ordered graph $G$ is not a semi-caterpillar, it contains one of the edge-ordered paths $ P_4^{213} $, $ P_5^{1342} $, or $ P_5^{1432} $. In the latter two cases we have $\ex_<(n,G)=\Omega(n\log n)$ since $\ex_<(n,P_5^{1432})=\Theta(n\log n)$ and $\ex_<(n,P_5^{1342})=\Omega(n\log n)$ were proved in Theorems~4.10 and 4.9 of \cite{GMNPTV} (see Theorem~\ref{34edgepath} below). We will therefore assumes that $G$ contains $P_4^{213}$.

The order chromatic number of the reverse $G^R$ of $G$ is also 2 and $G^R$ is not a semi-caterpillar either, so (again by Theorem~\ref{semi}) $G^R$ contains one of the edge-ordered paths $ P_4^{213} $, $ P_5^{1342} $, or $ P_5^{1432} $. In the latter two cases $\ex_<(n,G)=\ex_<(n,G^R)=\Omega(n\log n)$, so we assume $G^R$ contains $P_4^{213}$, that is, $G$ contains $P_4^{132}$.
 
We remark here that the disjoint union of $P_4^{213}$ and $P_4^{132}$ is a (disconnected) edge-ordered graph whose extremal function is linear. Therefore, we need to use again that $G$ is connected, the fact that it contains both $P_4^{213}$ and $P_4^{132}$ is not enough.

Let $G'$ be a bipartition of $G$ with all right vertices close. Such a bipartition exists by Lemma~\ref{ocn2}. Clearly, $G'$ must contain one of the bipartitions $P_4^{+213}$ or $P_4^{-213}$ of $P_4^{213}$. The common vertex of edges labeled 1 and 3 in $P_4^{213}$ is not close, but it is a right vertex in $P_4^{+213}$, so $G'$ must contain $P_4^{-213}$. Similarly, $G'$ must contain either $P_4^{+132}$ or $P_4^{-132}$, but the latter contains a right vertex that is not close, so $G'$ must contain $P_4^{+132}$.

Let $G''$ be the other bipartition of $G$. Clearly, $G''$ contains $P_4^{+213}$.

Lemma~4.11 of \cite{GMNPTV} claims that there exist edge-ordered bigraphs on $n$ vertices with $\Theta(n\log n)$ edges avoiding both $P_4^{+132}$ and $P_4^{+213}$. Clearly, these edge-ordered bigraphs avoid both $G'$ (as it contains $P_4^{+132}$) and $G''$ (as it contains $P_4^{+213}$). But then the edge-ordered graphs underlying this sequence of graph must avoid $G$ as the connected edge-ordered graph $G$ has only these two bipartitions. This shows $\ex_<(n,G)=\Omega(n\log n)$, as needed.
\end{proof}

\section{Edge-ordered paths}\label{sec4}

In this section, we estimate the extremal functions of edge-ordered paths. Let us first observe that Theorem~2.3 in \cite{GMNPTV} (the Erd\H os-Stone-Simonovits theorem for edge-ordered graphs) states that $\ex_<(n,G)=\Theta(n^2)$ whenever the order chromatic number of the edge-ordered graph $G$ is more than 2. The theorem even gives the exact asymptotics in terms of the order chromatic number. Therefore, we will concentrate on finding the order chromatic number of paths and estimating the extremal function of order chromatic number 2 paths. Note also that Lemma~\ref{semi-right} makes it very easy to determine if the order chromatic number of a path is 2.

While Theorem~2.3 in \cite{GMNPTV} gives only a very weak upper bound for the extremal function of order chromatic number 2 edge-ordered graphs, namely $o(n^2)$,  in the more recent paper \cite{KT} we give a stronger bound $\ex_<(n,G)=n2^{O\left(\sqrt{\log n}\right)}$ in case $G$ is an order chromatic number 2 edge-ordered forest. But the same paper formulates the even stronger conjecture that $\ex_<(n,G)=O(n\log^cn)$ holds for these edge-ordered graphs for some $c=c(G)$. This conjecture is wide open even for edge-ordered paths. Note however, that the corresponding conjecture for vertex-ordered graphs has recently been refuted, see \cite{PettieT}. Here we concentrate on the cases where such a bound can be established for an edge-ordered path $G$.

The MSc thesis of the first author, \cite{Kthesis} contained an elegant direct proof of a dichotomy result for edge-ordered paths, characterizing those with linear extremal functions and stating that the extremal functions of the remaining edge-ordered paths are $\Omega(n\log n)$. We state this result below as Corollary~\ref{main} and give a simple derivation from Theorem~\ref{dich}.

The edge-ordered path is called a \emph{monotone path} if the labels increase (or decrease) mo\-not\-o\-nous\-ly along the path. Note that the fact that monotone paths have linear extremal functions was established by R\"odl, \cite{R} well before edge-ordered graphs and their extremal functions were formally introduced. We call an edge-ordered path with at least $ 3 $ edges \emph{flipped}, if it is obtained from a monotone path by switching either the two smallest or the two largest labels. The fact that the extremal function of a flipped path is linear can also be deduced from Lemma~\ref{add} below that was already implicit in \cite{GMNPTV}.

Recall from the previous section that we simply denote the edge-ordered path $ P_k^L $ by listing the labels along the path in the upper index. For example, $ P_4^{123} $ is a 3-edge monotone path. A $k$-edge flipped path is isomorphic to either $P_{k+1}^{12\dots(k-2)k(k-1)}$ or $P_{k+1}^{2134\dots(k-1)k}$. Recall also that bipartitions of edge-ordered paths are denoted by including a sign $+$ or $-$ in the upper index signifying whether the path starts at a right or a left vertex.

\begin{corollary}[Dichotomy result on extremal function of edge-ordered paths]\label{main}
    Let $ P $ be an edge-ordered path with at least two edges.
    \begin{enumerate}
        \item If $ P $ is monotone or flipped, then $ \ex_<(n, P) = \Theta(n) $, 
        \item In all other cases, $ \ex_<(n, P) = \Omega(n \log n) $.
    \end{enumerate}
\end{corollary}

\begin{proof}
Let $P^R$ be the edge-ordered path obtained from $P$ by reversing the edge-order. We need to prove that $P$ or $P^R$ is a semi-caterpillar of order chromatic number 2 if and only if $P$ is monotone or flipped and then the statement of the corollary follows from Theorem~\ref{dich}. By Lemma~\ref{semi-right} semi-caterpillars of order chromatic number 2 are exactly the edge-ordered graphs underlying right caterpillars. By Lemma~\ref{equivalence}, right caterpillars are exactly those edge-ordered bigraphs that can be obtained from $T_0$ by a sequence of extensions. 

If the underlying simple graph of a right caterpillar is a path, then this must also hold for all intermediate edge-ordered bigraphs that those extensions create because if a vertex of degree at least $3$ appears at some point in the process then it cannot disappear and the resulting graph is not a path. Therefore, we need to study sequences of extensions of $T_0$ that create edge-ordered bigraphs with underlying simple paths.

Let $T$ be an edge-ordered bigraph with a path as its underlying simple graph and let $e$ be the minimal edge in $T$. When extending $T$ we add new edges at the ends of $e$. So the underlying simple graph remains a path if we add no new edge at any end of $e$ unless it is one of the two degree 1 vertices in $T$ and even then we add at most one new edge there. Thus, $T_0$ has three such extensions obtained by adding a single new edge at the left end of the only edge, or at the right end or at both ends. This way we obtain $P_3^{+12}$, $P_3^{-12}$, or $P_4^{+132}$, respectively.

Any edge-ordered bigraph whose underlying simple graph is a path longer than a single edge has at most a single extension whose underlying simple graph is still a path: we have to extend the path at the end where its smallest edge is with a single edge which is even smaller. (In case the smallest edge is an intermediate edge there is no such extension.) Starting from $P_3^{+12}$ or $P_3^{-12}$ we obtain bipartitions of monotone paths, while starting from $P_4^{+132}$ we obtain bipartitions of flipped paths, where the order of the largest two edges is flipped. So an edge-ordered path is an order chromatic number 2 semi-caterpillar if and only if it is a monotone path or a flipped path with the two highest labeled edges flipped.

The reverse edge-ordered path $P^R$ is an order chromatic number 2 semi-caterpillar if and only if $P$ is a monotone path or a flipped path with the bottom two edges flipped. One of $P$ or $P^R$ is an order chromatic number 2 semi-caterpillar if and only if $P$ is a monotone or flipped path. This finishes the proof of the corollary.
\end{proof}

Gerbner et al.\ estimated the extremal function of short edge-ordered paths in \cite{GMNPTV}. Here is a summary of their results.

\begin{theorem}[\cite{GMNPTV}]\label{34edgepath}
\begin{itemize}
\item The extremal functions $\ex_<(n,P_4^{123})$, $\ex_<(n,P_4^{132})$, $\ex_<(n,P_5^{1234})$, and $\ex_<(n,P_5^{1243})$ are all $\Theta(n)$.
\item The extremal functions $\ex_<(n,P_5^{1324})$, $\ex_<(n,P_5^{1432})$, and $\ex_<(n,P_5^{2143})$ are all $\Theta(n\log n)$.
\item  $\ex_<(n,P_5^{1342})=\Omega(n\log n)$ and $\ex_<(n,P_5^{1342})=O(n\log^2n)$.
\item $\ex_<(n,P_5^{1423})=\ex_<(n,P_5^{2413})=\binom n2$.
\end{itemize}
\end{theorem}

Note that the theorem lists all edge-ordered paths of three or four edges up to isomorphism or reversing the edge-order and it gives the order of magnitude for the corresponding extremal functions except for $P_5^{1342}$ where the upper bound is $\Theta(\log n)$ times the lower bound.

In this section we extend these results to 5-edge edge-ordered paths. We summarize our bounds on the extremal functions of edge-ordered 5-edge paths in Table~1. Some of the bounds listed there follow directly from observations in \cite{GMNPTV}, but the bound on $\ex_<(n,P_6^{13254})$, which is our main result in this section (Theorem~\ref{13254}), requires new techniques. Our list is complete, it contains all edge-ordered 5-edge paths up to isomorphism or reversing the edge-order, but for some of them, like $P_6^{21453}$ the best upper bound for the extremal function comes from the results in \cite{KT} and is very far from the lower bound of $\Omega(n\log n)$.

\begin{center}
    \begin{longtable}{ |c|c|c|c|c| } \hline
        Labeling ($ L $) & $\chi_<(P^L_6)$ & \multicolumn{2}{c|}{Bounds for $ \ex_<(n,P_6^L) $} & Proved in \\ \hline
        $ 12345, 12354 $ & $2$ & \multicolumn{2}{c|}{$ \Theta(n) $} & Prop.~4.7\cite{GMNPTV}\\ \hline
        $ 12435$ & \multirow{2}{*}{$2$} & \multicolumn{2}{c|}{\multirow{2}{*}{$  \Theta(n\log n) $}}  & Lemma 3.2, Theorem~4.12 \cite{GMNPTV} \\
        $15432$, $ 21543 $,  $ 12543 $ &  & \multicolumn{2}{c|}{} & Theorem~4.2 \cite{GMNPTV} \\ \hline
        $ 12453 $ & \multirow{2}{*}{$2$} & \multirow{2}{*}{$ \Omega(n\log n) $} & \multirow{2}{*}{$ O(n\log^2 n) $} & Theorem~4.14\cite{GMNPTV}\\
        $ 13254 $ &  &  &  & Theorem~\ref{13254}\\ \hline
        $ 14523 $, $ 14532  $, $ 15423 $, $ 21453 $ & $2$ & $ \Omega(n\log n) $ & $ n\cdot2^{O\left(\sqrt{\log n}\right)} $ & Theorem 2~\cite{KT}\\ \hline
        $ 14325 $ & \multirow{2}{*}{$3$} & \multicolumn{2}{c|}{\multirow{2}{*}{$ n^2/4+o(n^2) $}} & Theorem 4.15~\cite{GMNPTV}\\
        $ 21354 $ &  & \multicolumn{2}{c|}{} & Theorem~\ref{21354} \\ \hline
        $ 15243 $, $ 15234 $ & \multirow{5}{*}{$\infty$} & \multicolumn{2}{c|}{\multirow{5}{*}{$ \binom{n}{2} $}} & \multirow{5}{*}{Theorem 2.3~\cite{GMNPTV}} \\
        $ 24513 $, $ 25413 $, $ 15324 $, $ 21534 $ &  & \multicolumn{2}{c|}{} & \\
        $ 13524 $, $ 23514 $, $ 25143 $, $ 24153 $ &  &  \multicolumn{2}{c|}{} & \\
        $ 14253 $, $ 12534 $, $ 15342 $, $ 14352 $ &  & \multicolumn{2}{c|}{} & \\ 
        $ 13425 $, $ 13452 $, $ 13542 $, $ 25314 $ &  & \multicolumn{2}{c|}{} & \\\hline    \end{longtable}
    {{\bf Table 1.} Estimates for the extremal functions of the edge-ordered 5-edge paths.}
\end{center}

The first line of attack in determining the extremal function of any edge ordered path is to determine its order chromatic number. Lemma~\ref{semi-right} makes it really easy to decide if the order chromatic number of an edge-ordered path is 2. Theorem~2.4 in \cite{GMNPTV} gives a way to determine the order chromatic number in general. Deciding if the order chromatic number is finite is still relatively easy, but for intermediate values of the order chromatic number the procedure gets impractical really fast. The order chromatic number of the 4-vertex edge ordered graph $D_4$ given in \cite{GMNPTV} is still unknown. Therefore, we give no reference for the order chromatic number in Table~1 if it is 2 or infinite, but give a reference in the two cases when it is an intermediate value (namely 3). The version of the Erd\H os-Stone-Simonovits theorem for edge-ordered graphs (Theorem~2.3 in \cite{GMNPTV}) gives the exact value of the extremal function if the order chromatic number is infinite and its exact asymptotics if the order chromatic number is finite but larger than 2. Therefore we concentrate on the order chromatic number 2 case.

Corollary~\ref{main} gives a linear or $\Omega(n\log n)$ lower bound for every edge-ordered path. The linear bounds are tight and unfortunately no lower bound is known for any edge-ordered path (or forest) of order chromatic number 2 above $\Omega(n\log n)$. Thus, we only give the source of the upper bounds on the extremal functions for order chromatic number 2 paths in Table~1.

\medskip

In the rest of this section we state and prove the various bounds stated in Table~1. We start with establishing the order chromatic numbers of the edge-ordered paths. As mentioned above deciding whether that value is 2 is trivial using Lemma~\ref{semi-right} and deciding if it is infinite is easy using Theorem~2.4 in \cite{GMNPTV}. The exhaustive list in Table~1 contains only two edge-ordered 5-edge paths with an intermediate order chromatic number, $P_6^{14325}$ and $P_6^{21354}$. The order chromatic number of the former was established to be 3 in Theorem~4.15 of \cite{GMNPTV}. Here we give a similar argument proving the same for the latter one.

\begin{theorem}\label{21354}
    $\chi_<(P^{21354}_6) = 3$.
\end{theorem}
\begin{proof}
Neither end of the middle edge in $P_6^{21354}$ is close, so by Lemma~\ref{semi-right} we have $\chi_<(P_6^{21354}) > 2 $. In order to prove $ \chi_<(P_6^{21354}) \le 3 $, by Theorem~2.4 of \cite{GMNPTV}, it is enough to show that all canonical edge-orders of $ K_{3 \times 3} $ (the complete three-partite graph with three vertices in each of its three classes) contain $ P_6^{21354} $. There are many canonical edge-orders, but we do not even need their exact definition. For us, it is enough to know the fact about them stated in the following paragraph. This fact is stated in \cite{GMNPTV} and used in the proof of Theorem~4.15 there.
		
Let us denote the vertices  of $ K_{3\times 3} $ by $ u_1, u_2, u_3, v_1, v_2, v_3, w_1, w_2, w_3 $ in such a way that the edges are $ u_iv_j $, $ u_iw_j $ and $ v_iw_j $ for $ 1 \le i ,  j \le 3 $. Every canonical edge order of $ K_{3\times 3} $ can be obtained (up to isomorphism)  from a labeling of $K_{3\times3}$ in which $ L(v_iw_j) = 3i+j $ for every $i$ and $j$ such that these are the $ 9 $ smallest labeled edges or the 9 largest labelled edges. As reversing the edge-order of the path $ P_6^{21354} $ yields an isomorphic edge-ordered path, any edge-ordered graph contains $ P_6^{21354} $ if and only if its reverse order contains it. Therefore it is enough to consider the canonical edge-labeling $ L $ of $ K_{3 \times 3} $ having $ L(v_iw_j)=3i+j $ as the 9 smallest labels.
	
We can see that the paths $ w_2v_1w_1v_2 $ and $ w_2v_1w_1v_3 $ are both isomorphic to $ P^{213}_4 $. If the edge $ v_2u_1 $ is labeled higher than the edge $ v_3u_1 $ then the path $ w_2v_1w_1v_2u_1v_3 $ is isomorphic to $ P_6^{21354} $. Otherwise the path $  w_2v_1w_1v_3u_1v_2 $ is isomorphic to $ P_6^{21354} $. So $P_6^{21354} $ is contained in every canonical labeling of $K_{3\times3}$ and therefore $ \chi_<(P_6^{21354}) \le3 $.
\end{proof}

We turn now to the labelings of $ P_6 $ that have order chromatic number $ 2 $. The following lemma will prove useful in many cases. It was used implicitly in \cite{GMNPTV} but did not appear there explicitly.

\begin{lemma}\label{add}
	Consider a non-trivial edge-ordered graph $ H $. Let $ H' $ be the graph obtained by adding a single new edge connecting one end of the smallest edge in $H$ to a new vertex outside $H$ and making this new edge smaller then any edge in $H$. Then $ \ex_<(n,H') = \ex_<(n,H) + O(n) $.
\end{lemma}
\begin{proof}	 
Let $k$ be the number of vertices in $H$ and let $ G $ be an edge-ordered graph on $ n $ vertices that avoids $ H' $ and has the maximal number $ m = \ex_<(n,H') $ of edges. For every vertex $ v $ of $ G $, remove the smallest $ k $ edges incident to $ v $ (or all incident edges if the degree of $ v $ is less than $ k $), and denote by $G'$ the subgraph so obtained.
		
Clearly, $ G' $ has at least $ m - kn $ edges. It avoids $ H $ because otherwise we could extend a copy of  $ H $ with one of the deleted edges to obtain a copy of $ H' $ in $G$. We can choose any one from the $k$ smallest edges incident to the appropriate vertex as long as its other end is outside the copy of $H$ we are extending.
		
So $ m - kn \le \ex_<(n,H) $ which implies $ \ex_<(n,H') \le \ex_<(n,H) + O(n) $. In the other direction we have $ \ex_<(n,H') \ge \ex_<(n,H) $ because $ H' $ contains $ H $.
\end{proof}

From Lemma~\ref{add} and Theorem~\ref{34edgepath}, we get the upper bounds in Table~1 on the extremal functions of $P_6^{12543}$, $P_6^{12435}$, and $P_6^{12453}$. The linear bounds for monotone and flipped paths stated in Corollary~\ref{main} also follow.

Lemma~4.2 of \cite{GMNPTV} states that if $P$ is the concatenation of two monotone paths such that all edges in one of them is smaller than any edge in the other, then $\ex_<(n,P)=O(n\log n)$. This result applies to $P_6^{15432}$ and $P_6^{21543}$ and gives the bound stated in Table~1. It also applies to $P_6^{12543}$ that was already covered in the previous paragraph too.
	
Five edge-ordered 5-edge paths remain (up to isomorphism and reversing the edge order) that have order chromatic number 2, but for which the techniques of \cite{GMNPTV} do not seem to apply directly. These are $P_6^{13254}$, $P_6^{14523}$, $P_6^{14532}$, $P_6^{15423}$, and $P_6^{21453}$. For the extremal function of four of these edge-ordered paths, the best bound is $n\cdot2^{O\left(\sqrt{\log n}\right)}$ coming from \cite{KT}. This bound is much weaker than the $O(n\log^c n)$ bound conjectured in the same paper. We could, however, verify the conjecture for $P_6^{13254}$ (see Figure~\ref{fig3}):

\begin{figure}
	\centering
	\fbox{\includegraphics{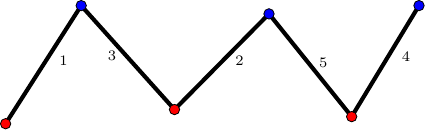}}
	\caption{$ P_6^{13254} $.}
	\label{fig3}
\end{figure}

\begin{theorem}\label{13254}
    $ \ex_<(n, P_6^{13254}) = O(n \log^2 n) $.
\end{theorem}

We will prove this theorem through a series of lemmas. We will use several proof techniques from the proof of Theorem~\ref{dich} in the previous section. A few of these have their origin in \cite{GMNPTV} like breaking symmetry by  considering edge-ordered bigraphs, or splitting a large edge-ordered graph in smaller parts according to the values of the labels. (We use lower and upper halves here as in \cite{GMNPTV}). 

Let $ G $ be an edge-ordered bigraph on $ n $ vertices having $ m $ edges. Throughout this section, we assume that the edge ordering in $ G $ is given by the labeling $ L $. In the rest of this section, we will refer to an edge of $G$ by first writing its left end, so the left end of the edge $xy$ is $x$, its right end is $y$. For a vertex $ v $ of $ G $, define the vertex label $ l(v) $ to be the second lowest label incident to it (vertices of degree at most $ 1 $ are not assigned a label). 
	
We define an edge $xy$ of $G$ to be \emph{left inclined} if $ l(x) < l(y) \le L(xy) $, and we say that $xy$ is \emph{right inclined} if $ l(y) < l(x)\le L(e) $. The edges that are neither left inclined nor right inclined are said to be \emph{non-inclined}. Denote the subgraph of $ G $ formed by its left inclined edges and all vertices of $G$ by $ G_l $ and the subgraph formed by the right inclined edges and all vertices of $G$ by $ G_r $.
	
The notions of left inclined and right inclined edges are somewhat similar, but should not be confused with the notions of $c$-left-leaning and $c$-right-leaning used in the previous section. We are not using $c$-left-leaning and $c$-right-leaning in this section except for calling the attention to the differences between these notions in this paragraph. While edges can be both $c$-left-leaning and $c$-right-leaning at the same time, left inclined edges are never right inclined, and vice versa. Another difference is that any $c$-left leaning edge of a subgraph is also $c$-left-leaning in the entire edge-ordered bigraph, but the analogous statement fails for left inclined edges. As for similarities, note how Lemmas~\ref{2cn}, \ref{obs} and their proofs are almost identical.

\begin{lemma}\label{obs}
	Let $ G $ be an edge-ordered bigraph on $ n $ vertices. Then $ G $ has at most $ 2n $ non-inclined edges.
\end{lemma}
\begin{proof}
	If an edge $xy$ is not among the two smallest edges at either of its ends, then we have $l(x)<L(xy)$ and $l(y)<L(xy)$. Further, $l(x)$ is the label of an edge not incident to $y$ so it must be distinct from $l(y)$. Therefore either $l(x)<l(y)<L(xy)$ or $l(y)<l(x)<L(xy)$ must hold and $xy$ cannot be non-inclined.
		
	Thus, all non-inclined edges of $G$ are among the two smallest edges at some vertex. This proves the lemma.
\end{proof}
	
\begin{lemma}\label{lem1}
	Let $ G $ be an edge-ordered bigraph avoiding $ P_6^{+13254} $, then $ G_l $ avoids $ P_5^{-2143} $.
\end{lemma}
\begin{proof}
	Let us assume for contradiction that $ xyx'y'x'' $ is a path in $ G_l $ isomorphic to $ P_5^{-2143} $, so $ x, x', x'' $ are left vertices and $ y, y' $ are right vertices. As $xy$ is left inclined in $G$, we have $ l(x) < l(y)$ (the vertex labels come from $G$, not $G_l$). As $x'y$ is also left inclined we have $l(y) \le L(x'y) $. Therefore, the two smallest edges of $G$ at $x$ are both smaller than $x'y$. One of these edges may be $ xy' $ that would create a cycle, but we can extend the path $ xyx'y'x'' $ with the other to get a path $ zxyx'y'x'' $ that is isomorphic to $ P_6^{+13254} $. This contradicts our assumption that $ G $ avoids $ P_6^{+13254} $ and the lemma follows.
\end{proof}
\begin{corollary}\label{co}
	Let $ G $ be an edge-ordered bigraph avoiding $ P_6^{-13254} $, then the subgraph $ G_r $ of $ G $ avoids $ P_5^{+2143} $.
\end{corollary}
\begin{proof}
	The statement follows from Lemma \ref{lem1} by symmetry.
\end{proof}
	
\begin{lemma}\label{lem2}
	If an edge-ordered bigraph $ G $ avoids $ P_6^{-13254} $ and $ P_5^{-2143} $, then it has $ O(n \log^2 n) $ edges.
\end{lemma}
\begin{proof}
	Let $ G $ be an edge-ordered bigraph avoiding $ P_6^{-13254} $ and $ P_5^{-2143} $ on $ n $ vertices with the maximal number $m$ of edges. Consider the subgraph $ A $ formed by the $ \lfloor m/2 \rfloor $ smallest edges of $ G $ (the lower half) and let $ B $ be the subgraph formed by the remaining edges of $ G $ (the upper half). Let $ A' $ be the subgraph of $ A $ consisting of those edges that are not among the two highest edges in $ A $ incident to any vertex. Similarly, let $ B' $ be the subgraph of $ B $ consisting of those edges that are not among the two smallest edges in $ B $ incident to any vertex. (All the subgraphs we consider contain all vertices of $G$.)
		
Consider the graph $ G' = A' \cup B' $ consisting of the edges belonging to $A'$ or $B'$ and partition it into the subgraphs $ G'_l $; $ G'_r $ and the remaining part formed by the non-inclined edges. As a subgraph of $G$, $ G'$ also avoids $P_6^{-13254}$, so by Corollary~\ref{co} $G'_r $ avoids $ P_5^{+2143} $. But $G'_r$ is also a subgraph of $G$, so it must avoid $P_5^{-2143}$ too. Therefore the edge-ordered graph underlying $G'_r$ avoids $ P_5^{2143} $. By Theorem~\ref{34edgepath} this underlying edge-ordered graph (and thus $ G'_r $ also) has $ O(n \log n) $ edges.

Assume there exists a left vertex $ x $ of $ G' $ that is incident to some edge from $A'$ and also to some edge from $B'$. Say the edge $ xy $ is in $ A' $ and the edge $ xz $ is in $ B' $. Let $ uy $ be the highest labeled edge in $ A $ incident to $ y $. Recall, that $ uy $ is not in $ A' $, so we have $ L(xy) < L(uy) $. Let $ vz $ be one of the two lowest labeled edges in $ B $ incident to $ z $, choosing between them to ensure $ u \ne v $. As before, we have $ L(xz) > L(vz) $. As $ vz $ is an edge in $ B $ and $ uy $ is an edge in $ A $, we must also have $ L(uy) < L(vz) $. Therefore the path $ uyxzv $ is isomorphic to $ P_5^{-2143} $ contradicting our assumption that $G$ avoids $P_5^{-2143}$. The contradiction proves that no left vertex can be incident to edges in both $ A' $ and $ B' $.

We split $G'_l$ into the subgraphs $ A''=A'\cap G'_l $ and $B''=B'\cap G'_l$ formed by edges of $G'_l$ in the lower and upper halves, respectively. Assume there is a right vertex $ y $ that is incident to an edge $xy$ from $A''$ and also to another edge $x'y$ from $B''$. As $x'y$ is left inclined in $G'$ we have $ l(x') < l(y)$ and as $xy$ is also left inclined in $G'$ we have $l(y) \le L(xy) $, where $ l(\cdot) $ refers to the vertex labels in $ G' $ determining left and right inclined, so $ l(x') $ is the second lowest label of an edge in $ G' $ incident to $ x' $. From $l(x')<L(xy)$ we conclude that the smallest edge at $x'$ in $G'_l$ is smaller than $xy$. But $xy$ is from the lower half of $G'$, therefore $x'$ is incident to an edge also from $A'$. But the same left vertex $x'$ also incident to $x'y$ from $B'$, a contradiction. The contradiction proves that edges from $A''$ and $B''$ cannot meet at a right vertex. But they cannot meet at a left vertex either, as we have established above that even the edges from $A'$ and $B'$ do not meet at a left vertex.

At most $2n$ of the $ \lfloor m/2 \rfloor $ edges of $A$ were removed to create $ A' $, so it has at least $ \lfloor m/2 \rfloor - 2n $ edges. All of these show up in $A''$ except those that are non-inclined in $G'$ ($O(n)$ edges by Lemma~\ref{obs}) or are in $ G'_r $ ($O(n\log n)$ edges). Therefore, $A''$ has $ m/2  - O(n \log n) $ edges. Similarly, $B''$ also has $m/2-O(n\log n)$ edges.

Let $ f(N) $ be the maximum number of edges in an $ N $-vertex edge-ordered bigraph that avoids $ P_6^{-13254} $ and $ P_5^{-2143} $. As $ m $ was chosen to be maximal we have $ m=f(n) $. Since edges in $A''$ and $B''$ do not meet, either $ A'' $ or $ B''$ has at most $ \lfloor n/2 \rfloor $ non-isolated vertices and therefore at most $ f(\lfloor n/2 \rfloor) $ edges, proving $f(\lfloor n/2 \rfloor)\ge m/2 -O(n\log n)$. Using $m=f(n)$ this implies the recursion $ f(n) \le 2f(\lfloor n/2\rfloor) + O(n\log n) $, which, in turn, implies, $ f(n) = O(n \log^2 n) $ as needed.
\end{proof}
	
\begin{proof}[Proof of Theorem~\ref{13254}]
	Let $ G $ be an edge-ordered graph with $ n $ vertices and the maximal number $ m = \ex_<(n, P_6^{13254}) $ of edges that avoids $ P_6^{13254} $. Partition the vertex set of $ G $ to left and right vertices and consider the edge-ordered bigraph $ G' $ formed by the edges of $G$ between the left and the right vertices. This partition can be chosen in such a way that $ G' $ contains at least $ m/2 $ edges. Clearly, $G'$ avoids both $P_6^{+13254}$ and $P_6^{-13254}$.
		
	The subgraph $ G'_l $ of $ G' $ formed by its left inclined edges avoids $ P_5^{-2143} $ by Lemma~\ref{lem1}. As a subgraph of $G'$ it also avoids $P_6^{-13254}$, so by Lemma~\ref{lem2}, $ G'$ has $ O(n \log^2 n) $ left inclined edges. By symmetry, there are also $ O(n \log^2 n) $ right inclined edges in $G'$. By Lemma~\ref{obs}, $G'$ has $ O(n) $ non-inclined edges. So $ G' $ has $ O(n \log^2 n) $ edges in total, that is $ m/2 = O(n \log^2 n) $ which implies $ \ex_<(n, P_6^{13254}) =m= O(n \log^2 n) $.
\end{proof}

\section{Concluding remarks and open problems}\label{sec5}

\noindent{\bf A.} Semi-caterpillars and especially right caterpillars proved to be important for the dichotomy stated in Theorem~\ref{dich}. Here we list some of their properties that may be instructive for getting the right intuition about them. We did not use these observations to arrive to Theorem~\ref{dich}, so we omit their simple proofs.

1. If we insist that all pairs of consecutive edges are adjacent in an edge-ordered tree, then we obtain a sub-class of semi-caterpillars. Note that in this case the underlying simple graphs are (conventional) caterpillars: each vertex is at distance at most one from a single path. The fact that these edge-ordered trees have a linear extremal function can be deduced easily from Lemma~\ref{add} that was already implicit in \cite{GMNPTV}. Semi-caterpillars form a larger class of edge-ordered graphs, their underlying simple graphs are not necessarily caterpillars. In fact, any tree is a subgraph of the underlying simple graph of a semi-caterpillar. Right caterpillars are intermediate in this respect: their underlying simple graph does not have to be caterpillars (as seen in Figure~\ref{fig1}), but it has a single path with all vertices in distance at most two from it (and in particular, all right vertices in distance at most one).

2. If two consecutive edges $e_1<e_2$ in a right caterpillar are connected by a larger edge $e$, then the right end of $e$ must be incident to $e_2$, as otherwise (being incident to $e_1$ and $e$ but not to $e_2$) it would not be close. This condition is indeed enough to ensure that a bigraph with an underlying semi-caterpillar is a right caterpillar. 

\begin{figure}[h]
\centering
\fbox{\includegraphics[scale=0.8]{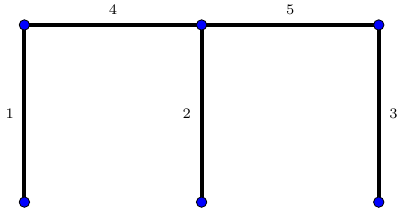}}
\caption{A semi-caterpillar whose subgraph (deleting the edge labeled 2) is not a semi-caterpillar.}
\label{fig4}
\end{figure}

3. A connected subgraph of a right caterpillar is a right caterpillar. Similar statement fails to hold for semi-caterpillars (see Figure~\ref{fig4}).

\medskip

\noindent{\bf B.} Theorem~\ref{dich} characterizes all \emph{connected} edge-ordered graphs with a linear extremal function. F\"uredi et al., \cite{FKMV} gave similar characterizations for connected vertex-ordered and convex geometric graphs. We find that it is worth to state a similar characterization (and dichotomy) also for 0-1 matrices. This setting is very closely related to vertex-ordered graphs. In fact, it directly corresponds to bipartite graphs whose two vertex classes are separately linearly ordered.

We say that a 0-1 matrix $A$ \emph{contains} another 0-1 matrix $B$ if $B$ can be obtained from a submatrix of $A$ by possibly decreasing some entries. Otherwise we say that $A$ \emph{avoids} $B$. The \emph{extremal function} $\eex(n,B)$ of a not-all-0 0-1 matrix $B$ is defined as
$$\eex(n,B)=\max\{\text{\#\ of 1 entries in }A\mid A\text{ is an }n\times n\text{ 0-1 matrix avoiding }B\}.$$

Characterizing 0-1 matrices with a linear extremal function seems to be beyond our reach (just as is the case for vertex-ordered graphs, convex geometric graphs or edge-ordered graphs), but it is relatively easy for \emph{connected} 0-1 matrices. We call a 0-1 matrix \emph{connected} if it is the bipartite adjacency matrix of a connected graph. See a conjecture for the characterization of \emph{light} 0-1 matrices with linear extremal functions in \cite{PettieT}. Light 0-1 matrices can be considered as the opposite of connected ones: we call a 0-1 matrix \emph{light} if each column contains a single 1.

We call a sequence of positions $(i_k,j_k)$ for $1\le k\le t$ in the integer grid a \emph{staircase} if $(i_{k+1},j_{k+1})$ is obtained from $(i_k,j_k)$, for any $1\le k<t$, by increasing exactly one of the two coordinates by one. We say that such a staircase \emph{describes} the $n\times m$ 0-1 matrix $A=(a(i,j))$ if \begin{align*}
&a(i_k,j_k)=1, \quad \text{for } 1\le k\le t, \\
&a(i_1,j)=1, \quad \text{for } 1\le j<j_1, \\
&a(i,j_1)=1, \quad \text{for } 1\le i<i_1, \\
&a(i_t,j)=1, \quad \text{for } b_t<j\le m, \\
&a(i,j_t)=1, \quad \text{for } i_t<i\le n, \\
&a(i,j)=0, \quad \text{otherwise}.
\end{align*}

We say that a 0-1 matrix $A$ is a \emph{staircase matrix} if either $A$ or the matrix obtained by reversing the order of the column in $A$ is described by a staircase. See Figure~\ref{fig5} for examples of staircase matrices.

\begin{figure}
\centering
\fbox{\includegraphics{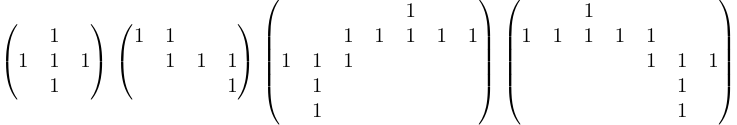}}
\caption{Staircase matrices. (Zeroes are omitted for clarity.)}
\label{fig5}
\end{figure}

Note that all staircase matrices are connected. In fact, each staircase matrix is the bipartite adjacency matrix of a caterpillar. We obtain the following characterization for the extremal function of connected 0-1 matrices.

\begin{theorem}\label{thm6}
\begin{enumerate}
\item The extremal function of a staircase matrix $A$ is linear: $\eex(n,A)=O(n)$.
\item If $A$ is a connected $0\text{-}1$ matrix but not a staircase matrix, then $\eex(n,A)=\Omega(n\log n)$.
\end{enumerate}
\end{theorem}

\begin{proof}[Sketch of proof]
F\"uredi and Hajnal, \cite{FH} define an \emph{elementary operation} on a 0-1 matrix $A$ as adding new  row or column to the boundary of $A$ and placing a single 1 entry in the new column or row directly next to an existing 1 entry in $A$. They observe that such an elementary operation increases the extremal function by at most an additive linear term. Part 1 of Theorem~\ref{thm6} follows from the observation that each staircase matrix can be obtained from the $1\times1$ matrix $(1)$ by repeated elementary operations. Note that $\eex(n,(1))=0$. 

Part 2 of Theorem~\ref{thm6} follows from the observation that each connected 0-1 matrix that is \emph{not} a staircase matrix either contains $\begin{pmatrix}1&1\\1&1\end{pmatrix}$ or one of the eight matrices that can be obtained from $\begin{pmatrix}1&1&0\\1&0&1\end{pmatrix}$ or its transpose by rotating them $0$, $90$, $180$ or $270$ degrees. Note that $\eex\left(n,\begin{pmatrix}1&1\\1&1\end{pmatrix}\right)=\Theta(n^{3/2})$, while $\eex\left(n,\begin{pmatrix}1&1&0\\1&0&1\end{pmatrix}\right)=\Theta(n\log n)$, \cite{Furedi}.
\end{proof}

The problem of determining $\eex(n,A)$ has been particularly extensively studied for so-called acyclic matrices $A$. Recently Janzer, Janzer, Magnan and Methuku \cite{JJMM} have proved tight bounds for non-acyclic matrices. 

\medskip

\noindent{\bf C.} We identified five essentially different edge-orderings of the 5-edge path $P_6$ in Section~\ref{sec4} whose extremal functions cannot be bounded by the results of \cite{GMNPTV}. Theorem~\ref{13254} gives an $O(n\log^2n)$ upper bound for the extremal function of one of them. For the other four, namely $P_6^{14523}$, $P_6^{14532}$, $P_6^{15423}$ and $P_6^{21453}$ the best upper bound for their extremal function comes from the general results in \cite{KT} and is of the form $n2^{O\left(\sqrt{\log n}\right)}$. The same paper formulated the stronger $n$ times polylog$(n)$ upper bound for all order chromatic number 2 edge-ordered forests as an ambitious conjecture. Proving this type of bound for some of those four remaining edge-orderings of $P_6$ would be a step towards verifying this conjecture.

Note, however that the analogous conjecture for 0-1 matrices (and hence for vertex ordered graphs) has been recently refuted, see \cite{PettieT}. In fact, Seth Pettie and the second author established that $\eex(n,S)=n2^{\Omega\left(\sqrt{\log n}\right)}$ for the matrix $S=\begin{pmatrix}1&0&1&0\\1&0&0&1\\0&1&0&1\end{pmatrix}$. $S$ was earlier identified in \cite{PT} as one of the two smallest matrices for which the conjecture bounding their extremal function by $O(n\log^cn)$ could not be verified and  Kor\'andi et al.~\cite{KTTW} proved $\eex(n,S)=n2^{O\left(\sqrt{\log n}\right)}$, a bound that was later shown to be tight in \cite{PettieT}.

Based on this analogy, we would not be shocked if our general conjecture on the extremal function of order chromatic number 2 edge-ordered forests failed or even if it failed for one of the above mentioned four edge-orderings of $P_6$. Right now, however, we do not know of a single edge-ordered forest of order chromatic number 2 whose extremal function is \emph{not} $O(n\log n)$. The first goal here is to prove a lower bound that beats this barrier.
\medskip

\noindent{\bf D.} We treat the extremal function problem largely solved for edge-ordered graphs for which we know the order of the magnitude of the extremal function. But further studies can also focus on finding the exact asymptotics for them. Even in the simplest case, namely the case of monotone paths, we do not have the leading constant in their linear extremal functions.
\medskip

\noindent{\bf E.} The edge-ordered variant of the Erd\H os-Stone-Simonovits theorem, see \cite{GMNPTV}, establishes the importance of the order chromatic number of edge-ordered graphs. This value is a positive integer or infinity. It was established in \cite{GMNPTV} that this value is two, three or infinity for edge-ordered paths of 3 or 4 edges. Here we established the same for edge-ordered paths of 5 edges. It would be interesting to figure out how high the order chromatic number of an edge-ordering of $P_k$ can be, if it is finite. A similar question was asked about general edge-ordered graphs on $k$ vertices in \cite{GMNPTV} and a triply-exponential upper bound was given, while, as a lower bound, edge-ordered graphs on $k$ vertices were constructed with finite, but exponential-in-$k$ order chromatic number. However the edge-ordered graphs in this example are not paths and we expect a lower value in this case. Specifically, we ask if the order chromatic number of an edge-ordering of $P_k$ is always \emph{polynomial} in $k$, if finite.

\end{document}